\documentclass[12pt]{amsart}

\usepackage{amsmath}
\usepackage{amssymb}
\usepackage{amsthm}
\usepackage{bm}
\usepackage{mathrsfs}
\usepackage{hyperref}
\usepackage{graphicx}
\usepackage{xcolor}
\usepackage{enumitem}
\setlist[enumerate]{parsep=0pt plus 4pt,topsep=0pt plus 4pt}
\usepackage{url}
\usepackage{mathtools}
\allowdisplaybreaks
\usepackage{epsfig}
\usepackage{xspace}

\newcommand\excise[1]{}

\newtheorem{thm}{Theorem}[section]
\newtheorem{lemma}[thm]{Lemma}
\newtheorem{prop}[thm]{Proposition}

\newtheorem{cor}[thm]{Corollary}

\newtheorem*{claim*}{Claim}
\theoremstyle{definition}
\newtheorem{defn}[thm]{Definition}
\newtheorem{remark}[thm]{Remark}

\numberwithin{equation}{section}

\newcounter{separated-sec}
\newcounter{collapse-sec}

\newcommand{\Ring}[1]{\ensuremath{\mathbb{#1}}}

\renewcommand\>{\rangle}
\newcommand\<{\langle}

\newcommand\KK{\mathcal{K}}
\newcommand\LL{\mathcal{L}}
\newcommand\MM{\mathcal{M}}

\newcommand\OO{\mathcal{O}}
\newcommand\RR{\Ring{R}}

\newcommand\XX{\mathcal{X}}

\newcommand\dd{\mathbf{d}}

\newcommand\pp{\mathfrak{p}}

\newcommand\cI{\mathcal{I}}

\newcommand\cS{\mathcal{S}}

\newcommand\ve{\varepsilon}
\newcommand\vS{\smash{\makebox[0pt][l]{$S$}%
                \raisebox{1.4ex}{\scalebox{1.1}[.5]{$^{\rightarrow\!\!}$}}}}%
\newcommand\vT{\smash{\makebox[0pt][l]{$T$}%
                \raisebox{1.4ex}{\scalebox{1.1}[.5]{$^{\rightarrow\!\!}$}}}}%

\newcommand\bmu{{\bar\mu}}

\newcommand\IZ{\cI(Z)}

\newcommand\SpM{S_p\MM}

\newcommand\TZX{\vT_Z\XX}

\newcommand\TpM{T_p\MM}

\newcommand\pOz{\pp_{\OO\to z}}

\renewcommand\implies{\Rightarrow}

\newcommand\interior[1]{{\kern0pt#1}^{\mathrm{o}}}

\DeclareMathOperator\CAT{CAT}

\DeclareMathOperator\hull{hull}

\DeclarePairedDelimiter\norm{\lVert}{\rVert}

\definecolor{teal}{HTML}{029386}
\begin{document}

\title{Shadow geometry at singular points of $\CAT(\kappa)$ spaces}
\author{Jonathan C. Mattingly}
\address{\rm Departments of Mathematics and of Statistical Sciences,
Duke University, Durham, NC 27708}
\urladdr{\url{https://scholars.duke.edu/person/jonathan.mattingly}}
\author{Ezra Miller}
\address{\rm Departments of Mathematics and of Statistical Sciences,
Duke University, Durham, NC 27708}
\urladdr{\url{https://scholars.duke.edu/person/ezra.miller}}
\author[Do Tran]{Do Tran}%
\address{\rm Georg-August Universit\"at at G\"ottingen, Germany}

\makeatletter
  \@namedef{subjclassname@2020}{\textup{2020} Mathematics Subject Classification}
\makeatother

%

\date{11 November 2023}

\begin{abstract}
In any $\CAT(\kappa)$ space~$\MM$, the \emph{shadow} of a tangent
vector $Z$ at a point~$p$ is the set vectors that form an angle
of~$\pi$ or more with~$Z$.  Taking logarithm maps at points
approaching~$p$ along a fixed geodesic ray from~$p$ with tangent~$Z$
collapses the shadow to a single ray while leaving isometrically
intact every convex cone that avoids the shadow of~$Z$.
\vspace{-2ex}
\end{abstract}

\maketitle
\tableofcontents
\vspace{-2.1ex}

\section*{Introduction}\label{s:intro}

\noindent
Singular spaces have in recent years attracted increasing interest
from
data science, where samples are taken from such a space~$\MM$ and the
goal is to carry out statistical analysis to the extent possible.
Examples include spaces of phylogenetic trees \cite{holmes2003, BHV01,
feragen-lo-et-al2013, lin-sturmfels-tang-yoshida2017,
lueg-garba-nye-huckemann2021}, shapes \cite{le2001,
kendall-barden-carne-le99}, and positive semi-definite matrices
\cite{groisser-jung-schwartzman2017, buet-pennec2023}, as well as from
studies of computer vision \cite{hartley-trumpf-dai-li2013} and
medical images \cite{pennec-sommer-fletcher2020,
pritchard-stephens-donnelly2000}.  Asymptotics of such sampling relies
on local geometry of~$\MM$ near the Fr\'echet mean of a probability
distribution~$\mu$ on~$\MM$.

When an empirical mean~$\bmu_n$ moves in a direction away from the
population Fr\'echet mean~$\bmu$, adding mass in any direction at an
angle of~$\pi$ or more from the direction at~$\bmu$ pointing
to~$\bmu_n$ drags the empirical mean directly toward~$\bmu$---not
approximately, but exactly.  Hence these \emph{shadow} directions may
as well all be collapsed to a ray for the purpose of determining the
effect of adding mass there on variation of that empirical mean.  That
is why the geometry of shadows is fundamental to geometric central
limit theorems on singular spaces in every known case:
\begin{itemize}
\item%
open books \cite{hotz-et-al.2013}, which are half-spaces of fixed
dimension~$d$ all glued along their boundary~$\RR^{d-1}$;
\item%
isolated planar hyperbolic singularities \cite{kale-2015}, which are
surfaces that are metrically except for one point at which the
curvature is negative; and
\item%
tree spaces \cite{barden-le2018} (see also the precursors
\cite{barden-le-owen2013, barden-le-owen2018} and the definition of
the relevant notion of tree space \cite{BHV01}).
\end{itemize}
These spaces are all $\CAT(0)$ polyhedral complexes, built from
metrically flat pieces by gluing in a way that creates only negative
curvature.  In contrast, central limit theorems have been known on
smooth manifolds for two decades \cite{bhattacharya-patrangenaru2003,
bhattacharya-patrangenaru2005}, and those theorems allow positive
curvature.

This paper is the first step in a program to prove central limit
theorems on singular stratified spaces, allowing positive curvature
bounded above and eliminating the need to glue flat pieces.  The
present goal, Theorem~\ref{t:isometry-limit-log}, is to prove that
given a point~$p$ in a $\CAT(\kappa)$ space~$\MM$, the shadow
directions for a given tangent vector $Z \in \TpM$ can be collapsed
with no effect on the geometry of convex cones in~$\TpM$ that avoid
the shadow: the collapse is an isometry on every such cone.

Constructing this collapse needs neither a measure on~$\MM$ nor a
stratification, smooth or otherwise; it uses only properties of
$\CAT(\kappa)$ spaces.
Therefore this paper isolates this purely geometric observation from
further arguments based on measure theory \cite{tangential-collapse}
and probability \cite{random-tangent-fields} that form the basis for
central limit theorems on stratified spaces \cite{escape-vectors}.

Basic definitions and elementary consequences surrounding
$\CAT(\kappa)$ spaces are gathered in Section~\ref{s:CAT(k)}.  The
theory of radial transport---that is, parallel transport away from the
apex in a~$\CAT(0)$ cone---occupies Section~\ref{s:radial}.  This
leads in Section~\ref{s:limit} to limit tangent spaces and limit
logarithms, which accomplish shadow collapse by taking logarithm maps
at points approaching the apex along a geodesic ray.

\vspace{-.5ex}
\subsection*{Acknowledgements}
DT was partially funded by DFG HU 1575/7.  JCM thanks the NSF RTG
\enlargethispage{.5ex}%
grant DMS-2038056 for general support.
\vspace{-.5ex}

\section{\texorpdfstring{$\CAT(\kappa)$}{CAT(k)} spaces}\label{s:CAT(k)}

\subsection{Angles and logarithm maps}\label{b:angles}
\mbox{}\medskip

\vspace{-.5ex}
\noindent
For more on $\CAT(\kappa)$ spaces, consult a metric geometry text,
such as \cite{BBI01}.

\begin{defn}[Injectivity radius]\label{d:inj-radius}
For any $\kappa \in \RR$, a \emph{model space of curvature~$\kappa$}
is a Riemannian manifold~$M_\kappa$ with geodesic distance
$\dd_\kappa$ and constant curvature~$\kappa$.  The \emph{injectivity
radius} of~$M_\kappa$ is $R_\kappa = \pi/\sqrt\kappa$ when $\kappa >
0$ and $R_\kappa = \infty$ if $\kappa < 0$.
\end{defn}

\begin{defn}[Comparison triangle]\label{d:comparison-triangle}
Given a triangle $\triangle xyz$ (a union of geodesics $x$ to~$y$
to~$z$ to~$x$) in a metric space $(\MM,\dd)$, a \emph{comparison
triangle} of $\triangle xyz$ in a model space $(M_\kappa,\dd_\kappa)$
is a triangle $\triangle x'y'z'$ in $M_\kappa$ such that
$\{x',y',z'\}$ is an isometric copy~of~$\{ x,y,z\}$.
\end{defn}

\begin{defn}[$\CAT(\kappa)$ metric space]\label{d:CATk}
A metric space $(\MM,\dd)$ is \emph{$\CAT(\kappa)$} if
\begin{enumerate}
\item%
any two points $x,y\in \MM$ such that $\dd(x,y) < R_\kappa$ can be
joined by a unique geodesic of length~$\dd(x,y)$; and
\item%
for any triangle $\triangle xyz$ in $\MM$ with $\dd(x,y) + \dd(y,z) +
\dd(z,x) < 2R_\kappa$, if $\triangle x'y'z'$ is a comparison triangle
in~$M_\kappa$ of~$\triangle xyz$, then the constant-speed geodesics
$\gamma: [0,1] \to \MM$ from $y$ to~$z$ and $\gamma': [0,1] \to
M_\kappa$ from $y'$ to~$z'$ satisfy, for all $t \in\nolinebreak{}
[0,1]$,
$$
  \dd\bigl(x,\gamma(t)\bigr) \leq \dd_\kappa\bigl(x',\gamma'(t)\bigr).
$$
\end{enumerate}
\end{defn}

\begin{defn}[Angle]\label{d:angle}
Let $\gamma_1\bigl([0,\ve_1)\bigr)$ and
$\gamma_2\bigl([0,\ve_2)\bigr)$ be two shortest geodesics
emanating from~$p$ in $(\MM,\dd)$, parametrized by arclength.  The
\emph{angle} between $\gamma_1$ and $\gamma_2$ is defined by
$$
  \cos\bigl(\angle(\gamma_1,\gamma_2)\bigr)
  =
  \lim_{t,s\to 0}\frac{s^2+t^2-\dd^2(\gamma_1(s),\gamma_2(t))}{2st},
$$
if the limit on the right exists.
\end{defn}

Angles between shortest paths exist in the presence of curvature
bounded above.

\begin{prop}\label{p:angles-exist}
If $(\MM,\dd)$ is a $\CAT(\kappa)$ space, then the angle between any
two shortest paths emanating from the same point exists.
\end{prop}
\begin{proof}
This is \cite[Exercise~4.6.3]{BBI01} with a proof similar to that of
\cite[Proposition~4.3.2]{BBI01}.
\end{proof}

\begin{defn}[Space of directions]\label{d:directions}
The \emph{space of directions} $\SpM$ at a point $p$ in a
$\CAT(\kappa)$ space $(\MM,\dd)$ is the set of equivalence classes of
shortest paths parametrized by arclength emanating from~$p$, where two
shortest paths are equivalent if the angle between them is~$0$.
\end{defn}

\begin{prop}\label{p:angular-metric}
There is a metric $\dd_s$ on $\SpM$ such that $(\SpM, \dd_s)$ is a
length space and $\dd_s(V, W) = \angle(V, W)$ for any $V, W \in
\SpM$ with $\angle(V, W) < \pi$, where $\angle(V, W)$ denotes the
angle between $V$ and~$W$.
\end{prop}
\begin{proof}
\cite[Lemma~9.1.39]{BBI01}.
\end{proof}

\begin{defn}[Anuglar metric]\label{d:angular-metric}
The \emph{angular metric} on~$\SpM$ is $\dd_s$ in
Proposition~\ref{p:angular-metric}.
\end{defn}

\begin{defn}[Tangent cone]\label{d:tangent-cone}
Suppose that $(\MM,\dd)$ is $\CAT(\kappa)$.  The \emph{tangent cone}
at a point $p \in \MM$ is
$$
  \TpM = \SpM \times [0,\infty) / \SpM \times \{0\},
$$
whose apex is usually also called~$p$ (although it can be called $\OO$
if necessary for clarity).  A~vector $W = W_p \times t$ with $W_p \in
\SpM$ has \emph{length} $\|W\| = t$ in~$\TpM$.
\end{defn}

\begin{defn}[Unit tangent sphere]\label{d:unit-tangent-sphere}
Elements in the \emph{unit tangent sphere} $\SpM$ of~$p$ in~$\MM$ are
identified with unit vectors in~$\TpM$: those of the form $V \times 1$
with $V \in \SpM$.
\end{defn}

\begin{defn}[Inner product]\label{d:inner-product}
Tangent vectors $V, W \in \TpM$ have \emph{inner product}
$$
  \<V,W\>_p = \|V\|\|W\| \cos\bigl(\angle(V,W)\bigr).
$$
The subscript $p$ is suppressed when the point~$p$ is clear from the
context.
\end{defn}

The angular metric $\dd_s$ induces a metric on the tangent cone
$\TpM$ which makes $\TpM$ a length space.
  
\begin{defn}[Conical metric]\label{d:conical-metric}
If~$\MM$ is a $\CAT(\kappa)$ space, $\TpM$ has \emph{conical metric}
$$
  \dd_p(V,W)
  =
  \sqrt{\|V\|^2 + \|W\|^2 - 2\<V,W\>}\
  \text{ for } V,W \in \TpM.
$$
\end{defn}

\begin{remark}\label{r:conical-metric}
The conical metric on $\TpM$ is analogous to the Euclidean metric on
tangent spaces of manifolds in the sense that, for any stratum~$R$
such that $p \in \overline{R}$, the space $(T_p R, \dd_p|_{T_p R})$ is
isometric to a closed subcone of some Euclidean space~$\RR^m$.
\end{remark}

\begin{defn}[Logarithm map]\label{d:log-map}
Fix a $\CAT(\kappa)$ space~$(\MM,\dd)$.  Let $\MM' \subseteq \MM$ be
the set of points with a unique shortest path to~$p$.  The
\emph{logarithm map} (or \emph{log map})
\begin{align*}
  \log_p : \MM' & \to \TpM
\\
              v &\mapsto \dd(p,v)V
\end{align*}
at~$p$ takes each point $v \in \MM'$ to the length $\dd(p,v)$ tangent
vector in the direction of the tangent $V \in \SpM$ to the unit-speed
shortest geodesic from~$p$ to~$v$.
\end{defn}

\begin{defn}[Conical $\CAT(0)$ space]\label{d:MM-is-conical}
A space~$\MM$ that is $\CAT(0)$ is \emph{conical} with apex~$p$ if the
log map $\MM \to \TpM$ is an isometry.
\end{defn}

\begin{remark}\label{r:conical}
The $\CAT(0)$ hypothesis in Definition~\ref{d:MM-is-conical} ensures
that $\log_p$ from Definition~\ref{d:log-map} is globally defined
on~$\MM$.
\end{remark}

\begin{remark}\label{r:exp}
The inverse of~$\log_p$, if it exists, is called the \emph{exponential
map} at~$p$.  It exists when~$\MM$ is a cone with apex~$p$, because in
that case $\log_p$ is an isometry, but in general $\log_p$ may not be
injective even in a small neighborhood of~$p$.  Exponential maps in
situations where $\log_p$ is not injective do not concern the
developments in this paper, but they are crucial for further
applications of the geometry here to central limit theorems; see
\cite[Example~3.7]{tangential-collapse}
for an illustration of local non-injecitivity of~$\log_p$.
\end{remark}
  
\subsection{Geometry of the tangent cone}\label{b:tangent-cone}
\mbox{}\medskip

\noindent
Collected here are properties of the tangent cone~$\TpM$ derived from
the presence of an upper bound on the curvature.

\begin{prop}[{\cite[Theorem~9.1.44]{BBI01}}]\label{p:tangent-cone-is-NPC}
If $\MM$ is a $\CAT(\kappa)$ space then
$(\TpM,\hspace{-.18ex}\dd_p)$ is a length space of nonpositive
curvature.
\end{prop}

\begin{remark}\label{r:curvature-boundedness}
Proposition~\ref{p:tangent-cone-is-NPC} is the reason why our theorems
have hypotheses bounding the curvature above: it naturally endows the
tangent cone~$\TpM$ with a conical metric~$\dd_p$
(Definition~\ref{d:conical-metric}) that is nonpositively
curved~(NPC).  Consequently, pushing the geometry of sampling forward
to~$\TpM$ yields a Fr\'echet function on an NPC space, which is convex
and thus \cite{sturm2003} has a unique mean.
\end{remark}

The first variation formula for the distance function from
\cite{BBI01} tells us that the distance function on $(\MM,\dd)$ has
first-order derivative.  The exponential $\gamma(t)$ in the statement
is a unit-speed geodesic whose tangent at~$q$ is~$V$.

\begin{prop}[First variation formula, \cite{BBI01}]\label{p:gradient-distance}
Let $p$ and $q$ be two different points in a $\CAT(\kappa)$
space~$\MM$.  Suppose that $V \in T_q\MM$ is a vector of unit length
in the tangent cone of~$q$, with exponential geodesic $\gamma(t) =
\exp_q(tV)$.  Then the function
\begin{align*}
  \ell: [0,1) & \to \RR
\\
            t &\mapsto \dd(p,\gamma(t))
\end{align*}
is right differentiable at $t = 0$ and
$$
  \ell'(0+)
  =
  \lim_{t\to 0}\frac{\ell(t)-\ell(0)}{t}
  =
  -\cos\angle(\log_qp,V).
$$
\end{prop}
\begin{proof}
See \cite[Theorem~4.5.6 and Remark~4.5.12]{BBI01}.
\end{proof}

Proposition~\ref{p:gradient-distance} implies that the angle function
on $\TpM$ is continuous if one vector is fixed.  However, it is
elementary that angles with a fixed basepoint---and hence inner
products---are continuous.

\begin{lemma}\label{l:inner-product-is-continuous}
The inner product function $\<\,\cdot\,,\,\cdot\,\>_p: \TpM \times
\TpM \to \RR$ for a fixed basepoint in a $\CAT(\kappa)$ space from
Definition~\ref{d:inner-product} is continuous.
\end{lemma}
\begin{proof}
The angle between pairs of unit vectors with a fixed basepoint is a
continuous function $\angle(\,\cdot\,,\cdot\,): \SpM \times \SpM \to
\RR$ by Proposition~\ref{p:angular-metric}, because distance functions
are continuous.
\end{proof}

\begin{remark}\label{r:angle-not-continuous}
The angle between pairs of vectors need not be continuous if the
basepoint is not fixed; see \cite[Section~4.3.3]{BBI01} for discussion
of this matter.  A~simple example to illustrate this behavior is the
plane~$\RR^2$ with open first quadrant removed.  Two perpendicular
rays~$\gamma_1$ and~$\gamma_2$ emanating from~$p$ converge to~$\OO x$
and~$\OO y$ as $p \to \OO$ along the negative part of the ray $x =
-y$; see Figure~\ref{f:angle-discont}.  While the angle between
$\gamma_1$ and~$\gamma_2$ is~$\pi/2$ before $p$ reaches~$\OO$, at the
limit the angle between~$\OO x$ and~$\OO y$ is~$\pi$.  This phenomenon
does not depend on there being a topological boundary: the whole
picture embeds in the kale \cite{kale-2015} with central angle~$5\pi/2$.
That is, nothing changes when two quadrants are glued onto the picture
in Figure~\ref{f:angle-discont}, one on the positive horizontal
$x$-axis and one on the positive vertical $y$-axis.
\end{remark}

\begin{figure*}[!ht]
\centering
  \includegraphics[width=0.35\textwidth]{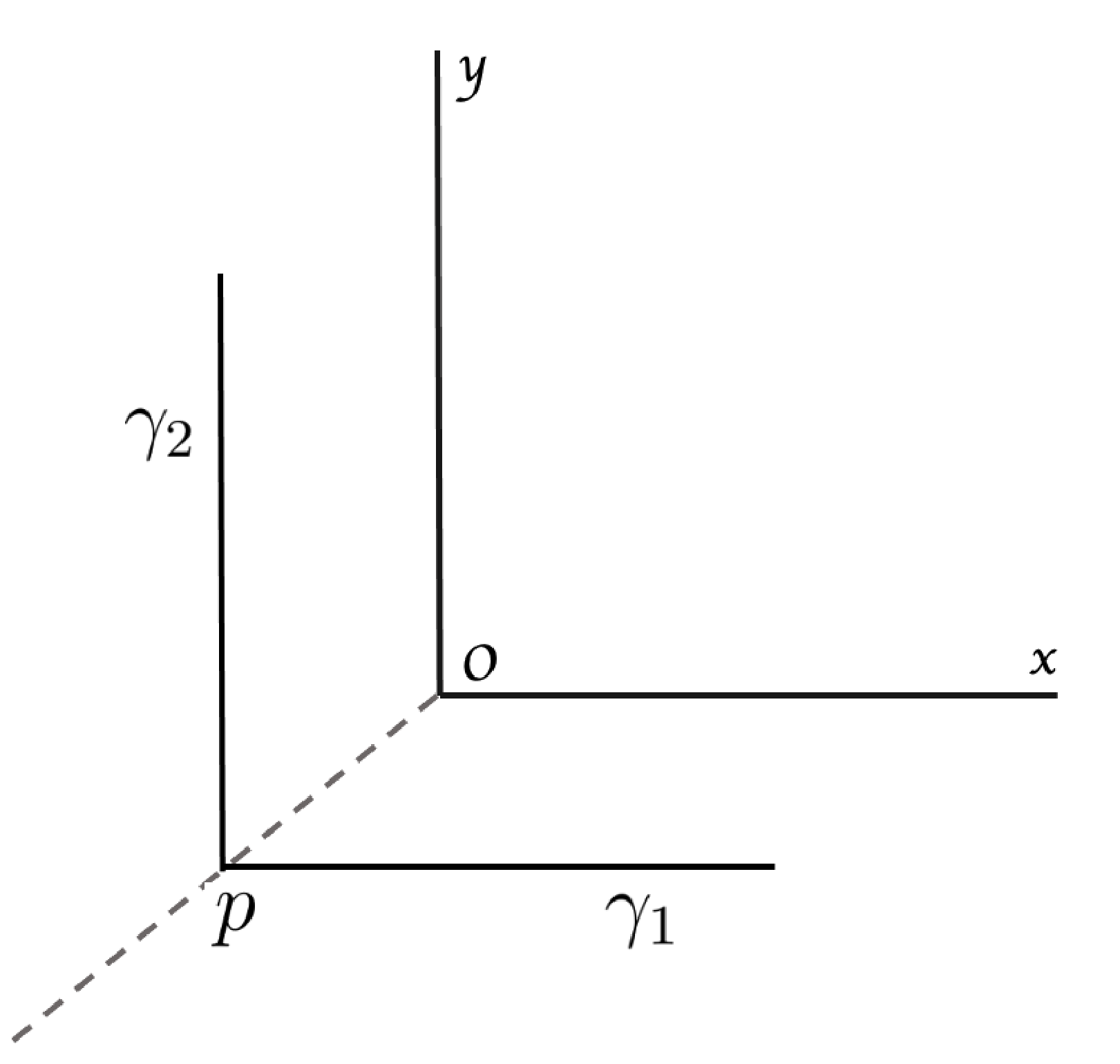}
\caption{Angle is not continuous.  The angle between $\gamma_1$ and
$\gamma_2$ is $\pi/2$ while the angle between $\OO x$ and $\OO y$ is $\pi$.}
\label{f:angle-discont}
\end{figure*}

A particularly useful property of $(\TpM,\dd_p)$ is that triangles in
$\TpM$ with one vertex at the apex are flat.

\begin{lemma}[Lemma~3.6.15, \cite{BBI01}]\label{l:flat-triangle}
Fix a $\CAT(\kappa)$ space~$\MM$.  For $V,W \in \TpM$ such that
$\angle(V,W) < \pi$, let $U(s), s\in [0,1]$ be a geodesic of constant
speed in $(\TpM,\dd_p)$ joining $V$ to~$W$.  Then $tU(s)$ for $s\in
[0,1]$ is a geodesic from~$tW$ to~$tV$ for all $t \in [0,1]$.  In
particular, any triangle in $\TpM$ with one vertex at the apex is
flat.
\end{lemma}
\begin{proof}
It suffices to show that, for $t,s\in [0,1]$,
$$
  \dd_p(tV,tW) = \dd_p\bigl(tV,tU(s)\bigr) + \dd_p(tU(s),tW).
$$
The above equality holds true for $t = 1$ because $U(s)$ for $s \in
[0,1]$ is a geodesic connecting~$V$ and~$W$.  Thus
$$
  \dd_p(V,W) = \dd_p\bigl(V,U(s)\bigr)+\dd_p\bigl(U(s),W\bigr)
  \text{ for } s \in [0,1].
$$
We now deduce, from the conical metric
(Definition~\ref{d:conical-metric}), that
\begin{align*}
\dd_p(tV,tW)
  & = \sqrt{\norm{tV}^2 + \norm{tW}^2 - 2t^2\norm{V}\norm{W}\cos\bigl(\angle(V,W)\bigr)}
\\
  & = t\sqrt{\norm{V}^2+\norm{W}^2 - 2\norm{V}\norm{W}\cos\bigl(\angle(V,W)\bigr)}
\\
  & = t\dd_p(V,W).
\end{align*}
Similarly, for $s\in [0,1],$
$$
  \dd_p\bigl(tV,tU(s)\bigr)
  =
  t\dd_p\bigl(V,U(s)\bigr)
  \quad\text{and}\quad
  \dd_p\bigl(tU(s),tW\bigr)
  =
  t\dd_p\bigl(U(s),W\bigr).
$$
Therefore
\begin{align*}
\dd_p(tV,tW)
  & = t\dd_p(V,W)
\\
  & = t\bigl(\dd_p(V,U(s)\bigr) + \dd_p\bigl(U(s),W)\bigr)
\\
  & = \dd_p\bigl(tV,tU(s)\bigr) + \dd_p\bigl(tU(s),tW\bigr),
\end{align*}
which completes the proof.
\end{proof}

This section concludes with two easy results that stand on their own
as generally useful but also arise in the intended application of this
theory to measures on smoothly stratified metric spaces; see
\cite[Lemma~2.29]{tangential-collapse},
for example.

\begin{prop}\label{p:TpM_is_NPC}
If $\MM$ is $\CAT(\kappa)$ and locally compact then $(\TpM, \dd_p)$
is~$\CAT(0)$: it is complete space simply connected, and globally
nonpositively curved (NPC).
\end{prop}
\begin{proof}
Because the metric $\dd_p$ on $\TpM$ is homogeneous (commutes with
scaling), any triangle can be brought to a similar one in any
neighborhood of the apex in~$\TpM$.  Thus, $\TpM$ is nonpositively
curved in the global sense (\cite[Definition 4.6.6]{BBI01}).  Due to
local compactness, $\TpM$ is complete since the space of directions is
complete.  Now invoke \cite[Remark~9.2.1]{BBI01} to conclude that
$(\TpM, \dd_p)$ is simply connected and hence a global NPC space.
\end{proof}

\begin{cor}\label{c:S-is-CAT(1)}
If $\MM$ is $\CAT(\kappa)$ and locally compact then $(\SpM, d_s)$
is~$\CAT(1)$.
\end{cor}
\begin{proof}
It follows from Proposition~\ref{p:TpM_is_NPC} that $(\TpM,\dd_p)$ is
a $\CAT(0)$ space.  As $(\TpM,\dd_p)$ is a Euclidean cone over
$(\SpM,\dd_s)$, invoke \cite[Theorem~4.7.1]{BBI01} to conclude that
$(\SpM,\dd_s)$ has curvature bounded above by~$1$.
\end{proof}

\section{Radial transport}\label{s:radial}

\noindent
Defining parallel transport on~$(\TpM,\dd_p)$ for $\CAT(\kappa)$
space~$\MM$ is difficult in general.  For a simple example, excise the
open first quadrant of the plane~$\RR^2$ as in
Figure~\ref{f:extension}; what results is a $\CAT(0)$ stratified space
whose tangent cone at the origin does not compare with any tangent
cones nearby.  However, not all is lost: parallel rays are still
defined on $\CAT(0)$ spaces.  Hence limits of tangent cones can be
taken upon approach to a singular point.  This way of comparing
tangent cones is enough to push the fluctuating cone through a
d\'evissage process and hence compare $\TpM$ to a vector space.
\begin{figure*}[!ht]
\centering
\includegraphics[width=0.3\textwidth]{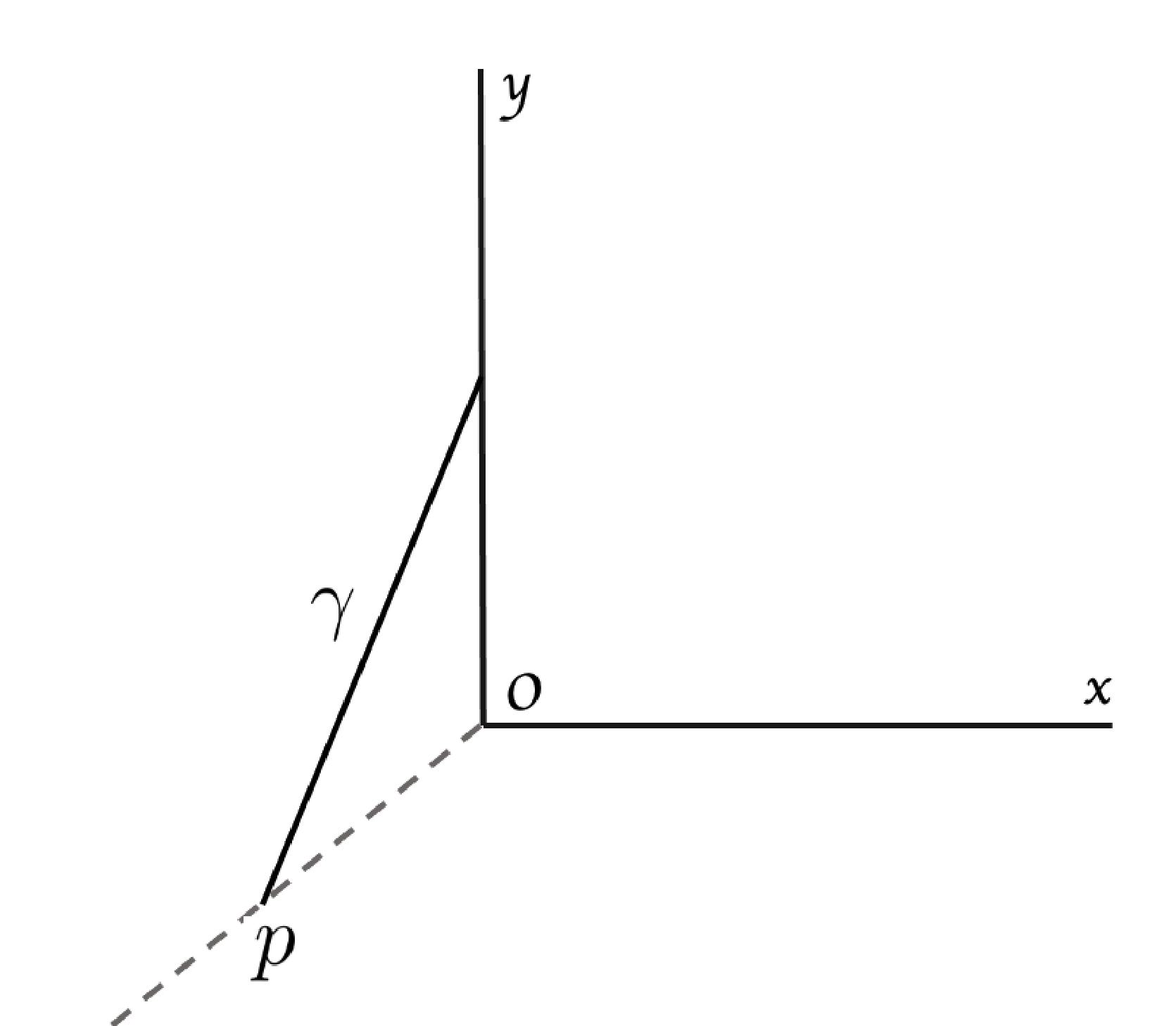}
\caption{The plane with open first quadrant excised.  The geodesic
segment $\gamma$ cannot be extended indefinitely, so no ray is
parallel to it.}
\label{f:extension}
\end{figure*}

The hypotheses for this section start with a $\CAT(\kappa)$
space~$\MM$ and its tangent cone $\XX = \TpM$,
whose apex may be called~$p$ but alternatively may be called~$\OO
\in\nolinebreak \XX$ to distinguish it from~$p \in \MM$.

Although parallel transport in~$\MM$ need not be well defined, only
parallel transport on~$\TpM$ along geodesics starting from the
apex~$p$ is required for our CLT.  (That explains the typical
hypothesis $\XX = \TpM$ in this section.)  This ``radial transport''
works because every triangle with a vertex at the cone point is flat
by Lemma~\ref{l:flat-triangle}.

First recall some background on parallel lines and rays in $\CAT(0)$
spaces such as~$\XX$ from \cite[Chapter~9]{BBI01}, where more details
can be found.

\begin{defn}[Lines]\label{d:line}
A \emph{line} in a $\CAT(0)$ space $\XX$ is a unit-speed geodesic
$\gamma: \RR \to \XX$ such that every closed subinterval of~$\gamma$
is a shortest path in~$\XX$.  A \emph{ray} in~$\XX$ is a half-line
geodesic $\gamma: [0,+\infty) \to \XX$.
\end{defn}

\begin{defn}[Parallel lines]\label{d:parallel-lines}
In a $\CAT(0)$ space, two unit-speed lines or rays $\gamma_1(t)$
and~$\gamma_2(t)$ are \emph{parallel}, written $\gamma_1 \parallel
\gamma_2$, if the function $f(t) =
\dd\bigl(\gamma_1(t),\gamma_2(t)\bigr)$ is bounded.
\end{defn}


\begin{lemma}[{\cite[Proposition 9.2.28]{BBI01}}]\label{l:existence-parallel-ray:BBI}
Fix a point $z$ in a $\CAT(0)$ space~$\XX$.  For any ray
$\gamma\bigl([0,+\infty)\bigr)$ in~$\XX$ there exists a unique ray
parallel to $\gamma$ starting at~$z$.
\end{lemma}

\begin{remark}\label{r:parallel}
Fix a $\CAT(\kappa)$ space~$\MM$ and $\XX = \TpM$ with apex~$\OO$.
Fix $z \in \XX$ and a point $q \neq \OO$ on the geodesic~$\OO z$
joining $\OO$ to~$z$.  Suppose that $V_z$ is a vector in~$T_z\XX$
whose exponential $\exp_z V_z$ is defined.  The triangle formed by
$\OO$, $z$, and~$\exp_z V_z$ is flat by Lemma~\ref{l:flat-triangle}.
Hence a unique point~$v_q$ lies on the geodesic from~$\OO$
to~$\exp_zV_z$ such that the geodesic segment $qw_q$ is parallel
to~$\exp_z V_z$ in the Euclidean sense; see~Figure~\ref{f:parallel}.
\begin{figure*}[!ht]
\centering
\includegraphics[width=0.3\textwidth]{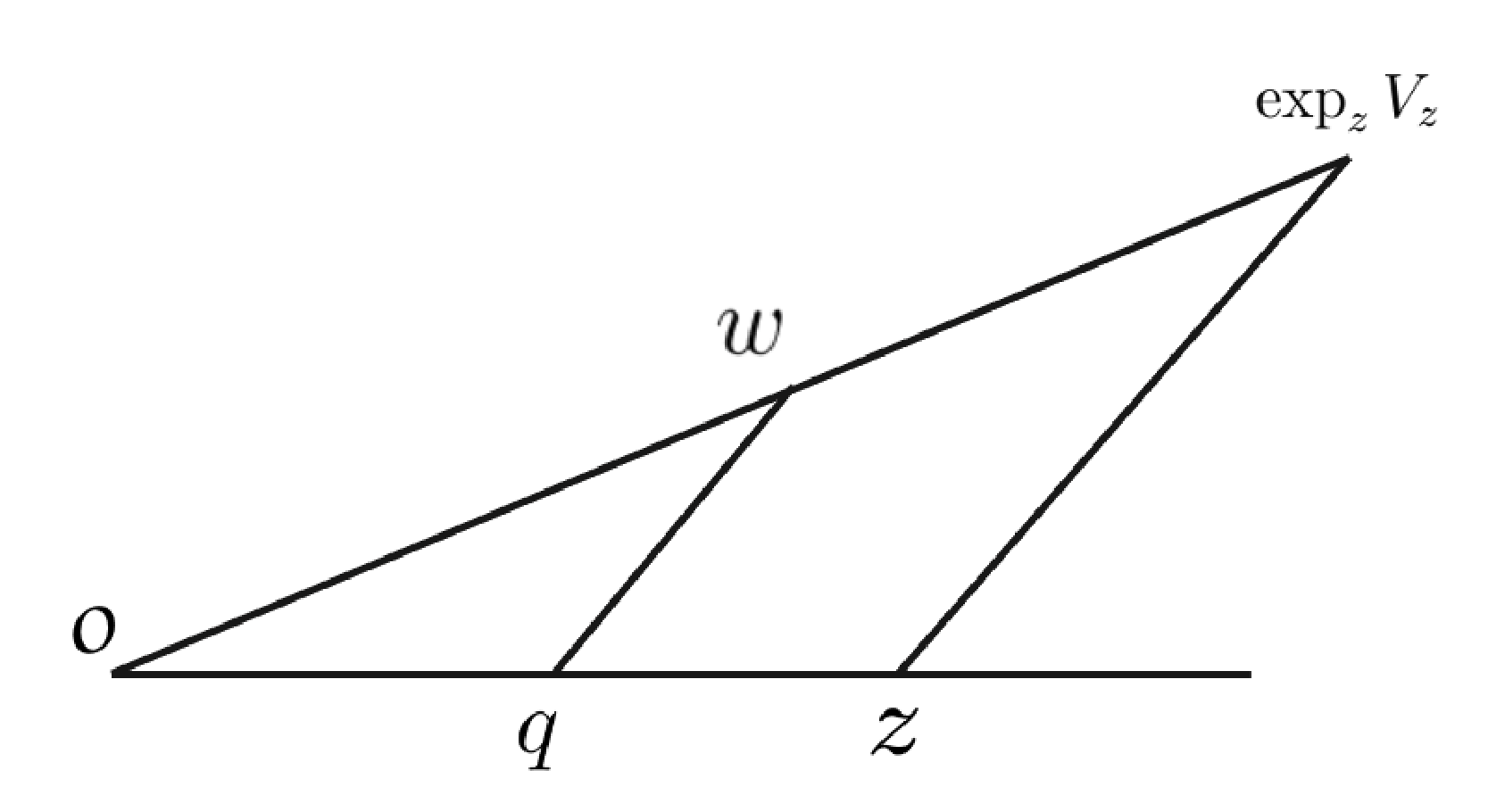}
\caption{A Euclidean (flat) sector in $\XX$ spanned by $\OO z$ and
$\OO v_z$ where $v_z = \exp_zV_z$.  For any non-apex $q$ on $\OO z$,
there is only one $v_q$ on~$\OO v_z$ such that $qv_q \parallel zv_z$
in the Euclidean sense.}
\label{f:parallel}
\end{figure*}
\end{remark}

\begin{defn}[Parallel vectors]\label{d:parallel-vectors}
The two vectors $V_q = \log_q v_q$ and $V_z$ in
Remark~\ref{r:parallel} are \emph{parallel}.  More generally, if $\XX
= \TpM$ with apex~$\OO$, and $q$ lies on the geodesic~$\OO z$
from~$\OO$ to~$z$ in~$\XX$, then two nonzero vectors $V_z \in T_z\XX$
and $V_q \in T_q\XX$ are \emph{parallel} if there exist $\ve_z > 0$
and $\ve_q > 0$ such that $\ve_z V_z$ and $\ve_q V_q$ are parallel.
\end{defn}

\begin{lemma}\label{l:uniqueness-parallel}
Let~$\MM$ be $\CAT(\kappa)$ and $\XX = \TpM$ with apex~$\OO$.  Fix $z
\in \XX$ and $q \in \OO z$.  Any nonzero vector $V_z\in T_z\XX$ has a
unique unit vector $V_q \in T_q\XX$ parallel~to~$V_z$.
\end{lemma}
\begin{proof}
The point $v_q$ in Remark~\ref{r:parallel} is unique.
\end{proof}


\begin{defn}[Radial transport]\label{d:radial-transport}
Let~$\MM$ be $\CAT(\kappa)$ and $\XX = \TpM$ with apex~$\OO$.  Fix $z
\in \XX$ and $\OO \neq q \in \OO z$.  For any unit vector $V_z\in
T_z\XX$ let $V_q$ be the unit vector in $T_q \XX$ parallel to~$V_z$
afforded by Lemma~\ref{l:uniqueness-parallel}.  The \emph{radial
transport} from~$q$~to~$z$~is
\begin{align*}
  \pp_{q \to z}: T_{q}\XX & \to T_z\XX
\\
                     tV_q &\mapsto tV_z \text{ for all } t\geq 0.
\end{align*}
The inverse of $\pp_{q \to z}$ is the radial transport $\pp_{z \to q}
= \pp_{q \to z}^{-1}$ (justified by
Proposition~\ref{p:radial-is-isometry}).
\end{defn}

\begin{prop}\label{p:radial-is-isometry}
Radial transport is an isometry: in
Definition~\ref{d:radial-transport}, $W_{\hspace{-1pt}z}
\hspace{-1pt}=\hspace{-1pt} \pp_{q \to z}(W_q) \hspace{-1pt}\implies$
$$
  \angle(V_z,W_z) = \angle(V_q,W_q).
$$
\end{prop}
\begin{proof}
Write $|xy|$ for the length of the geodesic segment from $x$ to~$y$
when the space containing $x$ and~$y$ is clear from context.  Given $s
\in (0,1)$ such that $s < |\OO q|$, write
\begin{equation}\label{eq:radial-is-isometry:1}
  v_p = \exp_p \bigl({\textstyle\frac{|\OO p|}{|\OO z|}}s V_p\bigr)
  \ \text{ and }\
  w_p = \exp_p \bigl({\textstyle\frac{|\OO p|}{|\OO z|}}s W_p\bigr)
\end{equation}
for $p \in \{q,z\}$; see Figure~\ref{f:radial}.
\begin{figure*}[!ht]
\centering
\includegraphics[width=0.25\textwidth]{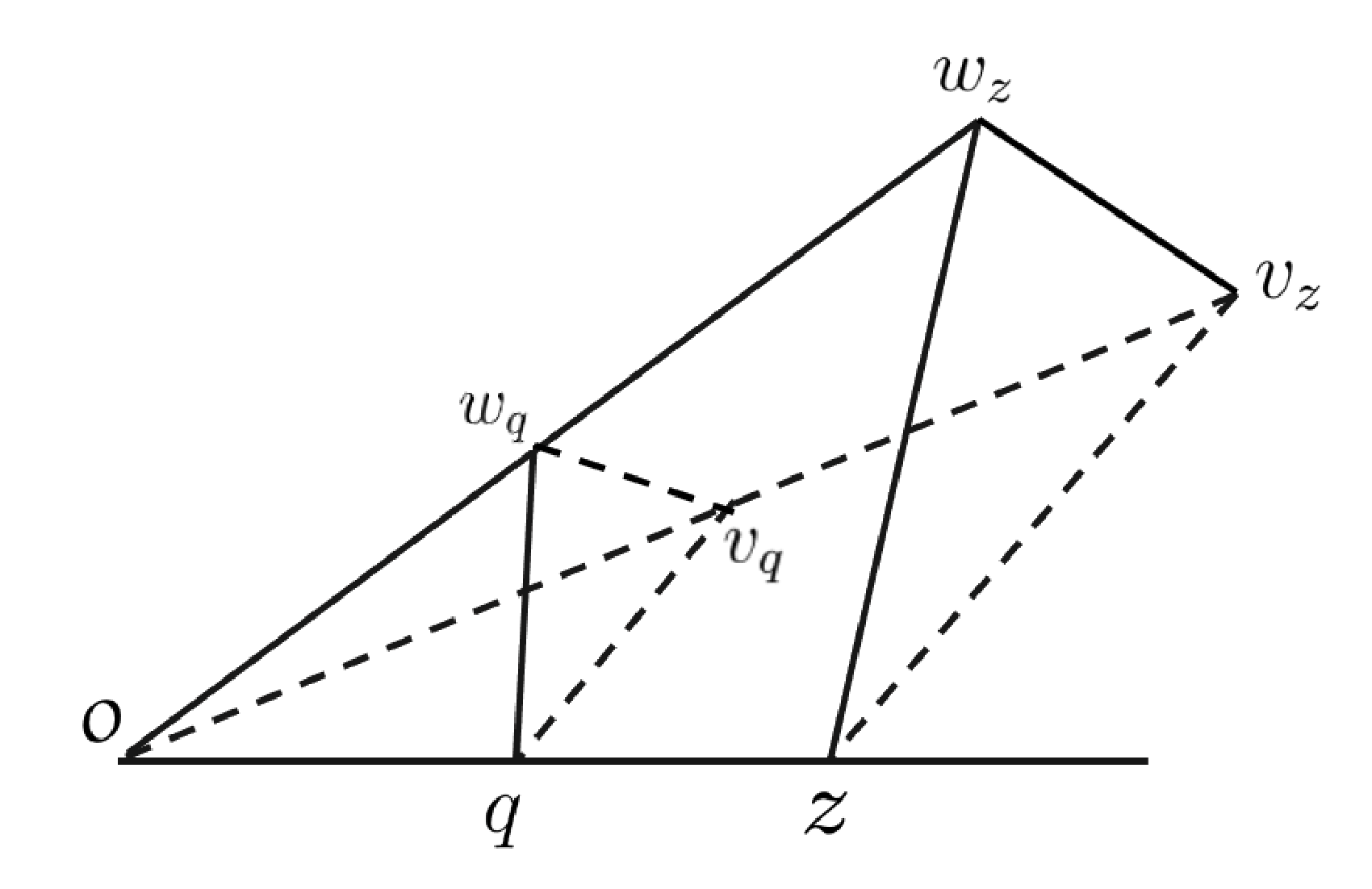}
\caption{The same setup as in Figure~\ref{f:parallel} but with two
flat sectors.  Here $v_q w_q \parallel v_z w_z$.}
\label{f:radial}
\end{figure*}
If any of the points $w_p$ or $v_p$ for $p \in \{q,z\}$ lie in the
geodesic through $\OO$ and~$z$, the condition $s < |\OO q|$ ensures
that those points lie in the ray~$\{\exp_\OO tz\}_{t \geq 0}$, so
radial transport applies.

The six triangles
$\triangle\OO p v_p$, $\triangle\OO p w_p$, and $\triangle\OO v_p w_p$
for $p \in \{q,z\}$ are flat by Lemma~\ref{l:flat-triangle}.  Since
$q v_q \parallel z v_z$ and $q w_q \parallel z w_z$, with both length
ratios being $\frac{|\OO q|}{|\OO z|}$ by construction,
$$
  \frac{|\OO v_q|}{|\OO v_z|}
  =
  \frac{|\OO w_q|}{|\OO w_z|}
  =
  \frac{|\OO q|}{|\OO z|}.
$$
As $\triangle \OO v_z w_z$ is also flat, it follows that $v_q w_q
\parallel v_z w_z$ and that
$$
  \frac{|v_q w_q|}{|v_z w_z|}
  =
  \frac{|\OO q|}{|\OO z|}.
$$
Hence the two triangles $\triangle q v_q w_q$ and $\triangle z v_z
w_z$ are similar, so $\angle(q v_q, q w_q) = \angle(z v_z, z w_z)$
independent of~$s$ in~\eqref{eq:radial-is-isometry:1}.  Letting
that~$s$ converge to~$0$ leads to the conclusion that $\angle(V_q,
W_q) = \angle(V_z, W_z)$, as desired.
\end{proof}

\begin{remark}
Proposition~\ref{p:radial-is-isometry} and the definition of radial
transport here differ from \cite[Theorem~2.11]{MBH15}.  In particular,
results stated in \cite[Theorem~2.11]{MBH15} rely on the definition,
at the beginning of page~4 in that paper, that parallel rays
$\gamma_1(t)$ and~$\gamma_2(t)$ have constant, uniform distance; that
is
$$
  \dd\bigl(\gamma_1(t),\gamma_2(t)\bigr) = C
$$
for some constant~$C$.  From this definition, \cite[Lemma~2.9]{MBH15}
claims that given a ray~$\gamma_1$ and a point~$p$ there always exists
a ray $\gamma_2$ starting from $p$ that is parallel to $\gamma_2$.
However, a counterexample to this claim is depicted in
Figure~\ref{f:no-ray}.
\end{remark}
\begin{figure*}[!ht]
\centering
\vspace{-2ex}
\includegraphics[width=0.45\textwidth]{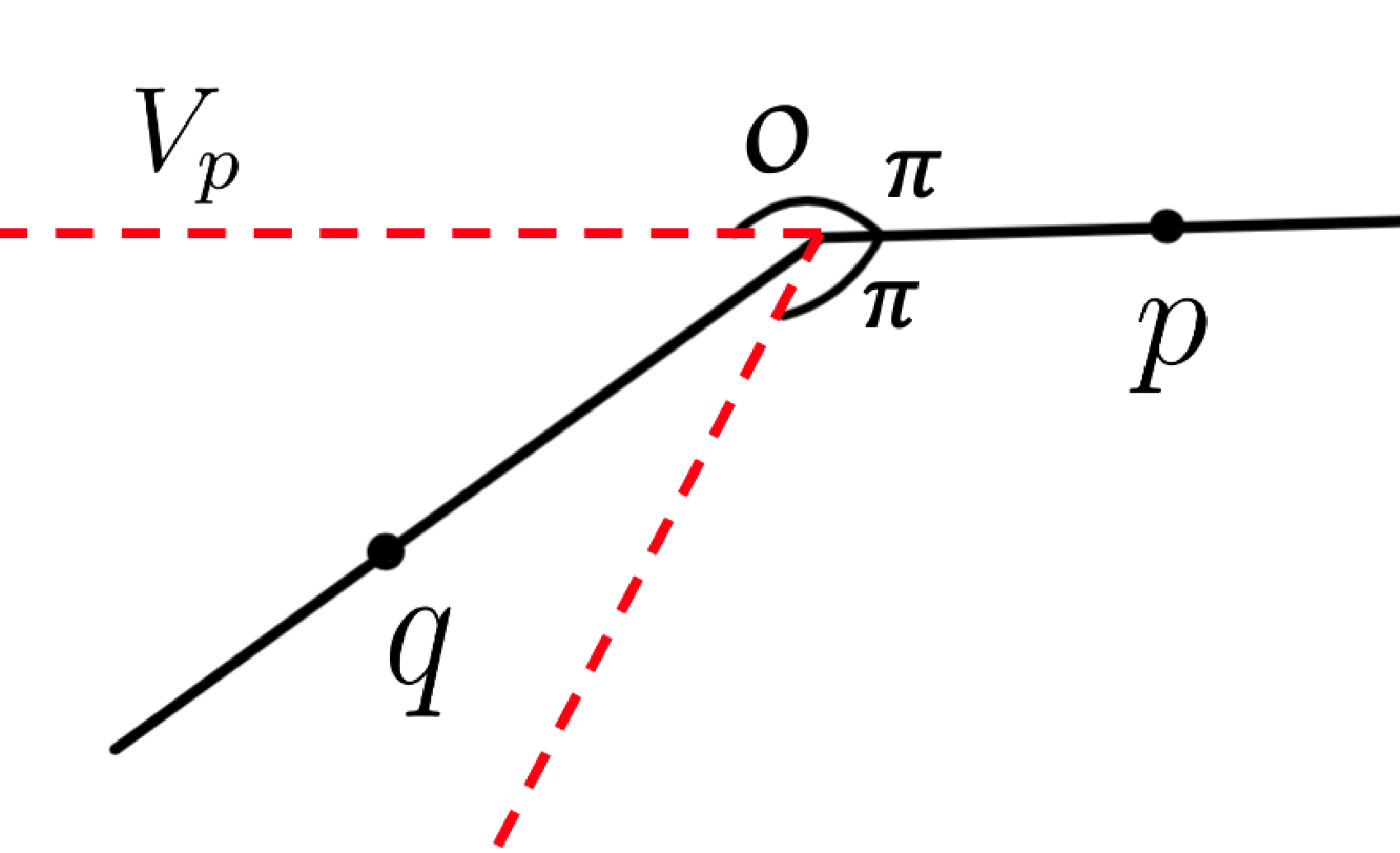}
\caption{On a topological plane with isolated singularity of total
angle~$3\pi$, the red dashes are the boundary of the shadow of~$\OO
p$.  If~$q$ is in the shadow and $V_p$ is the unit vector in the
direction~$p\OO$ then there is no ray that starts from~$q$ and is
parallel to~$V_p$ as defined in \cite{MBH15}.}
\label{f:no-ray}
\end{figure*}

Since the purpose of radial transport is to collapse by d\'evissage,
the next goal is to extend radial transport to the case where the
initial point is the apex~$\OO$.

\begin{remark}\label{r:parallel-at-cone-point}
Let~$\MM$ be $\CAT(\kappa)$ and $\XX = \TpM$ with apex~$\OO$.  Fix $z
\in \XX$.  Suppose that $V \in T_\OO\XX$ is tangent to the cone point.
The ray $\exp_\OO tV$ is well defined for $t\in [0,\infty)$ since
$\OO$~is the apex.  Lemma~\ref{l:existence-parallel-ray:BBI} produces
a unique ray $\gamma\bigl([0,\infty)\bigr)$ starting at~$z$ and
parallel to~$\exp_\OO tV$.  As~$\gamma\bigl([0,\infty)\bigr)$ is a
geodesic, it has a unit tangent $V_z \in T_z\XX$ that is thought of as
parallel to~$V$.  The precise definition follows.
\end{remark}

\begin{defn}[Parallel vectors, one at~$\OO$]\label{d:parallel-at-cone-point}
Let~$\MM$ be $\CAT(\kappa)$ and $\XX = \TpM$ with apex~$\OO$.  Fix $z
\in \XX$.  For any unit vector $V \in T_\OO\XX$, let $V_z \in T_z\XX$
be the unique unit vector such that the ray $\exp_\OO tV$ is parallel
to the ray $\exp_z tV_z$.  The vectors $tV$ and $tV_z$ are
\emph{parallel}, written $tV \parallel tV_z$, for all $t \in
[0,\infty)$.
\end{defn}

\begin{defn}[Radial transport from~$\OO$]\label{d:radial-transport-from-OO}
Let~$\MM$ be $\CAT(\kappa)$ and $\XX = \TpM$ with apex~$\OO$.  Fix $z
\in \XX$.  The \emph{radial transport} from~$\OO$ to~$z$ is the map
\begin{align*}
  \pOz: T_\OO\XX & \to T_z\XX
\\
              tV &\mapsto tV_z \text{ for all } t \geq 0
\end{align*}
in which $tV \parallel tV_z$ as in
Definition~\ref{d:parallel-at-cone-point}.
\end{defn}
\begin{figure*}[!ht]
\centering
  \includegraphics[width=0.35\textwidth]{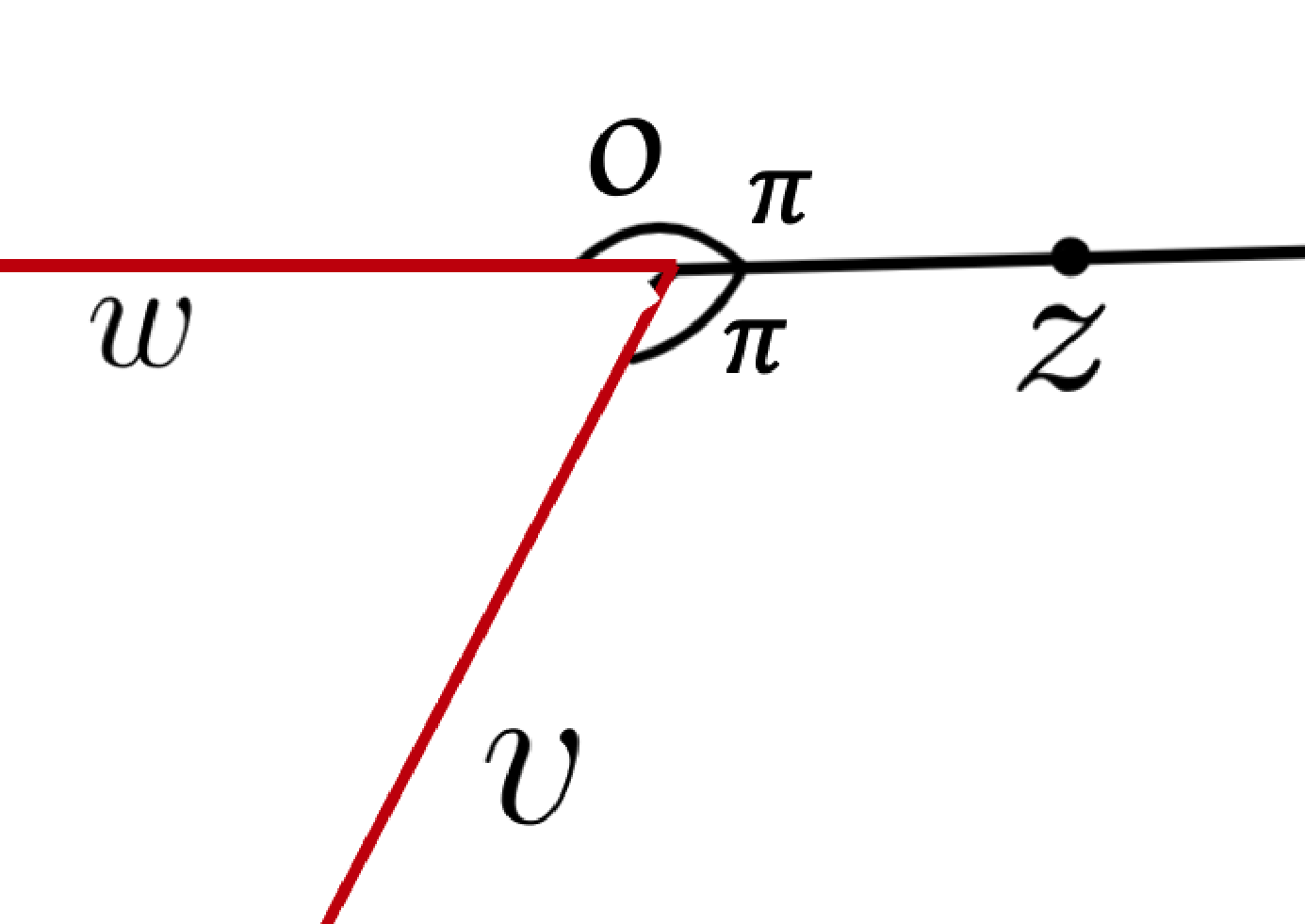}
\caption{Radial transport from $\OO$ to~$z$ is not isometric.  Both
$v$ and~$w$ are mapped to a vector in the direction from~$z$ to~$\OO$.}
\label{f:transport}
\end{figure*}

\begin{remark}\label{r:z-to-OO}
Transport from $T_z\XX$ to~$T_\OO\XX$ might not be possible because
rays starting from~$z$ may not extend indefinitely.  This failure is
related to Remark~\ref{r:radial-not-isom}.
\end{remark}

\begin{remark}\label{r:radial-not-isom}
Unlike radial transport between two non-apex points in
Definition~\ref{d:radial-transport}, radial transport $\pOz$ from the
apex~$\OO$ to a less singular point~$z$ need not be an isometry.  For
example, if $V,W\in T_\OO\XX$ have $\angle(V,\log_\OO z) =
\angle(W,\log_\OO z) = \pi$ (so $V$ and~$W$ lie in the shadow of~$z$;
see Figure~\ref{f:transport} and Definition~\ref{d:shadow}), then
$\pOz V = \pOz W$.
\end{remark}

Remark~\ref{r:radial-not-isom} notwithstanding, not all is lost where
the isometry is concerned; see
Proposition~\ref{p:almost-iso-of-singular-radial-transport}.  To explore how
much isometry survives, another notion is required.

\begin{defn}[Half-strip]\label{d:half-strip}
Let~$\MM$ be $\CAT(\kappa)$ and $\XX = \TpM$ with apex~$\OO$.
Parallel rays $\gamma_1, \gamma_2 \subseteq \XX$ \emph{span a convex
flat half-strip} if there are parallel rays $\zeta_1,\zeta_2$ in
$\RR^2$ and an isometry $\iota$ that maps the convex hull
of~$\gamma_1$ and~$\gamma_2$ in~$\XX$ to the convex hull of~$\zeta_1$
and~$\zeta_2$ in~$\RR^2$ satisfying $\zeta_1(t) =
\iota\bigl(\gamma_1(t)\bigr)$ and $\zeta_2(t) =
\iota\bigl(\gamma_2(t)\bigr)$ for all $t \in [0,\infty)$.
\end{defn}

The exponentials in the following result exist because they occur in a
conical space.

\begin{prop}\label{p:convex-half-strip}
Let~$\MM$ be $\CAT(\kappa)$ and $\XX = \TpM$ with apex~$\OO$.  Fix $z
\in \XX$ and a unit vector $V \in T_\OO\XX$.  Set $V_z = \pOz V \in
T_z\XX$.  Then $\exp_\OO tV$ and $\exp_z(tV_z)$ for $t \in [0,\infty)$
span a convex flat half-strip (whose width can be~$0$).
\end{prop}
\begin{proof}
It follows from the proof of \cite[Proposition~9.2.28]{BBI01} that the
ray $\exp_\OO tV$ is the limit of the geodesic from $\OO$ to $\exp_z(tV_z)$
as $t$ goes to infinity.  In other words,
$$
  V
  =
  \lim_{t\to\infty}\frac{\log_\OO\bigl(\exp_z(tV_z)\bigr)}
                        {\bigl\|\log_\OO\bigl(\exp_z(tV_z)\bigr)\bigr\|}.
$$
On the other hand, invoke Lemma~\ref{l:flat-triangle} again to see
that the convex hull of~$\OO$ and the ray $\exp_z(tV_z)$ is flat.
Hence the convex hull of $\exp_\OO tV$ and $\exp_z(tV_z)$ is also
flat.
\end{proof}

\section{Limit tangent spaces and limit logarithm maps}\label{s:limit}

\noindent
Radial transport compares tangent cones as they approach the
apex~$\OO$ along a geodesic.  Proposition~\ref{p:radial-is-isometry}
says that those tangent cones are isometric via radial transport.  It
is therefore natural to identify them in the limit, as follows.

\begin{defn}[Limit tangent cone]\label{d:limit-tangent-cone}
Let~$\MM$ be $\CAT(\kappa)$ and $\XX = \TpM$ with apex~$\OO$.  Fix $Z
\in T_\OO\XX$.  For $q, q' \in \OO z$ between $\OO$ and~$z = \exp_\OO
Z$, radial transport $\pp_{q \to q'}$ identifies $T_q\XX$ with
$T_{q'}\XX$.  The \emph{limit tangent cone} along~$Z$ is the
direct limit
$$
  \TZX = \varinjlim_{q \in \OO z} T_q\XX.
$$
Write $\vS_Z\XX$ for the unit sphere around the apex in~$\TZX$.
\end{defn}

\begin{remark}\label{r:dirlim}
The direct limit here is an algebraic or categorical notion rather
than an analytic one; see \cite[Chapter~III.10]{lang2002}.  An element
of $\TZX$ is represented by a tangent vector in~$T_q\XX$ at a point
beteween~$\OO$ and~$z$, and two such vectors---at different
points---represent the same limit tangent element if they are parallel
transports of each other.  The direct limit allows radial transport
from the apex (Definition~\ref{d:radial-transport-from-OO}) to be
viewed as comparing tangent data at~$\OO$ to tangent data
infinitesimally near~$\OO$ along the given direction~$Z$.  In
applications to smoothly stratified spaces~$\MM$
\cite{tangential-collapse}, the limit tangent space~$\TZX$ is
automatically less singular than $\XX = \TpM$ itself, in the precise
sense that the codimension of the singularity decreases upon taking
limit tangent spaces; see
\cite[Proposition~4.18.2]{tangential-collapse}.
Indeed, this is a key motivation for defining limit tangent spaces.
\end{remark}

\begin{remark}\label{r:orthant-space-limit-log}
Barden and Le define \emph{boundary limits} of \emph{translated
logarithm maps} in orthant spaces \cite[Theorem~2]{barden-le2018},
which accomplish what limit log maps do here.  Translation can
substitute in orthant spaces for the more general but weaker radial
transport
because orthant spaces are glued from pieces of Euclidean spaces
\cite{centroids}.
\end{remark}

The exponentials in the following result exist because they occur in a
conical space.

\begin{defn}[Limit log map]\label{d:limit-log}
Let~$\MM$ be $\CAT(\kappa)$ and $\XX = \TpM$ with apex~$\OO$.  Fix $Z
\in T_\OO\XX$.  For any $V \in T_\OO\XX$ and any $q \in \OO z$ between
$\OO$ and~$z = \exp_\OO Z$, write $V_q = \pp_{\OO\to q} V$.  Let $V_Z$
be the image of~$V_q$ (for any $q \neq \OO$) in the limit tangent
space~$\TZX$.  The \emph{limit log map} along $Z$~is
\begin{align*}
  \LL_Z: T_\OO\XX & \to \TZX
\\
             tV &\mapsto tV_Z \text{ for all } t \geq 0.
\end{align*}
\end{defn}

\begin{remark}\label{r:limit-log}
The usual log map at a point $p$ in a $\CAT(\kappa)$ space~$\MM$
takes each point $p \in \MM$ to the tangent at~$p$ of the geodesic
aimed at the terminus~$p$ as it exits the initial point~$p$.  In
contrast, the limit log map at~$p$ in the direction~$Z$ reflects
what happens when each point~$p$ goes to the tangent of the geodesic
aimed at the terminus~$p$ as it exits an initial point~$p'_Z$ that
is infinitesimally near~$p$ along~$Z$.  The limit log map is more
accurately the derivative of this mapping, in that it takes the
tangent vector pointing from~$p$ toward~$p$ to a tangent vector
at~$p'_Z$.  This information is recorded at~$p$ itself, rather
than at~$p'_Z$, via the direct limit in
Definition~\ref{d:limit-tangent-cone}.
\end{remark}

\begin{remark}\label{r:folding-map}
The limit log map was called the \emph{folding map} on a hyperbolic
topological plane with isolated singularity \cite{kale-2015} or on an
open book \cite{hotz-et-al.2013} because the limit log map along a
direction~$Z$ collapses rays whose directions are ``beyond opposite''
to~$Z$.  The general version is made precise in the next Definition
and Remark.
\end{remark}

\begin{defn}[Shadow]\label{d:shadow}
Let~$\MM$ be $\CAT(\kappa)$ and $\XX = \TpM$ with apex~$\OO$.  The
\emph{shadow} of a tangent vector $Z \in T_\OO\XX$ is the set $\IZ$ of
nonzero vectors that form an angle of~$\pi$ with~$Z$:
$$
  \IZ = \{V \in T_\OO\XX \mid \angle(V,Z) = \pi\}.
$$
\end{defn}

\begin{remark}\label{r:parallel-in-shadow}
When $V$ lies in the shadow of~$Z$, the ray through~$Z$ parallel
to~$\exp_\OO tV$ passes through~$\OO$ itself.  Hence, by
Definition~\ref{d:limit-log} (see also
Definitions~\ref{d:parallel-at-cone-point}
and~\ref{d:radial-transport-from-OO}), $\LL_Z(V)$ is a scalar multiple
of the vector~$\log_Z\OO$ that points from~$Z$ directly toward~$\OO$.
This vector arises numerous times and can be written in various ways,
such as
$$
  \log_Z\OO = -\LL_Z(Z) = -\pOz(Z)
$$
for $z = \exp_\OO Z$, all expressing that this vector is the unit
tangent at~$Z$ to the ray $Z \OO$ traversed backward along the
ray~$\OO Z$.  The scalar in question is
$\|\LL_Z(V)\|/\|\mathord{-}\LL_Z(Z)\|$.  In particular, if a sequence
of vectors in~$T_\OO\XX$ converges to a vector in the shadow of~$Z$,
then the sequence of images under the limit log map converges to this
same vector:
$$
  \frac{\LL_Z V}{\|\LL_Z V\|}
  \to
  \frac{-\LL_Z Z}{\|\LL_Z Z\|}
  \ \text{ as }\ V \to \IZ.
$$
The next result generalizes this observation that after taking the
limit log map along~$Z$, the shadow $\IZ$ becomes the exact
opposite vector of~$Z$.
\end{remark}

\begin{prop}\label{p:sum=pi}
Let~$\MM$ be $\CAT(\kappa)$ and $\XX = \TpM$ with apex~$\OO$.  Fix~$Z
\in\nolinebreak\! T_\OO\XX$.  For any $W_Z \in \TZX$,
$$
  \angle\bigl(\LL_Z(Z), W_Z\bigr) + \angle\bigl(W_Z, -\LL_Z(Z)\bigr)
  =
  \angle\bigl(\LL_Z(Z), -\LL_Z(Z)\bigr)
  =
  \pi.
$$
\end{prop}
\begin{proof}
Set $z = \exp_\OO Z \in \XX$.  By definition of $-\pOz(Z)$
in Remark~\ref{r:parallel-in-shadow},
$$
  \angle(\pOz Z, -\pOz Z) = \pi.
$$
By Lemma~\ref{l:flat-triangle}, the convex hull of $\OO W_Z$ and $\OO
z$ is flat, so
$$
  \angle(\pOz Z, W_Z) + \angle(W_Z, -\pOz Z)
  =
  \angle(\pOz Z, -\pOz Z)
  =
  \pi,
$$
and the desired result follows.
\end{proof}

Although radial transport from $\OO$ to~$z$ is not an isometry, it is
close to being one, in the sense that it retains isometric properties
away from the shadow.

\begin{prop}\label{p:almost-iso-of-singular-radial-transport}
Let~$\MM$ be $\CAT(\kappa)$ and $\XX = \TpM$ with apex~$\OO$.  Fix
unit vectors $V,W,Z \in T_\OO\XX$ with $\angle(V,W) < \pi$ such that
the geodesic $VW$ in~$T_\OO\XX$ does not intersect the shadow~$\IZ$.
Let $z = \exp_\OO Z$ and $U_z = \pOz U$ for $U \in \{V,W\}$.  Then
$$
  \angle(V,W) = \angle(V_z,W_z),
$$
and at the level of geodesics, $V_z W_z = \pOz(VW)$.
\end{prop}
\begin{proof}
Given $r \in (0,1]$, write
\begin{align*}
    v = \exp_\OO rV,\ &\ \ \ \,v_z = \exp_zrV_z,
\\  w = \exp_\OO rW,  &\ \ \   w_z = \exp_zrW_z.
\end{align*}
It follows from Proposition~\ref{p:convex-half-strip} that the rays
$\exp_\OO(tV)$ and $\exp_ztV_z$ span a convex flat half-strip.  Hence
$\OO,v,V_z,z$ constitute the vertices of a parallelogram (in $\RR^2$).
Similarly, $\OO,w,W_z,z$ make a parallelogram (see
Figure~\ref{f:prism}).
\begin{figure*}[!ht]
\centering
  \includegraphics[width=0.35\textwidth]{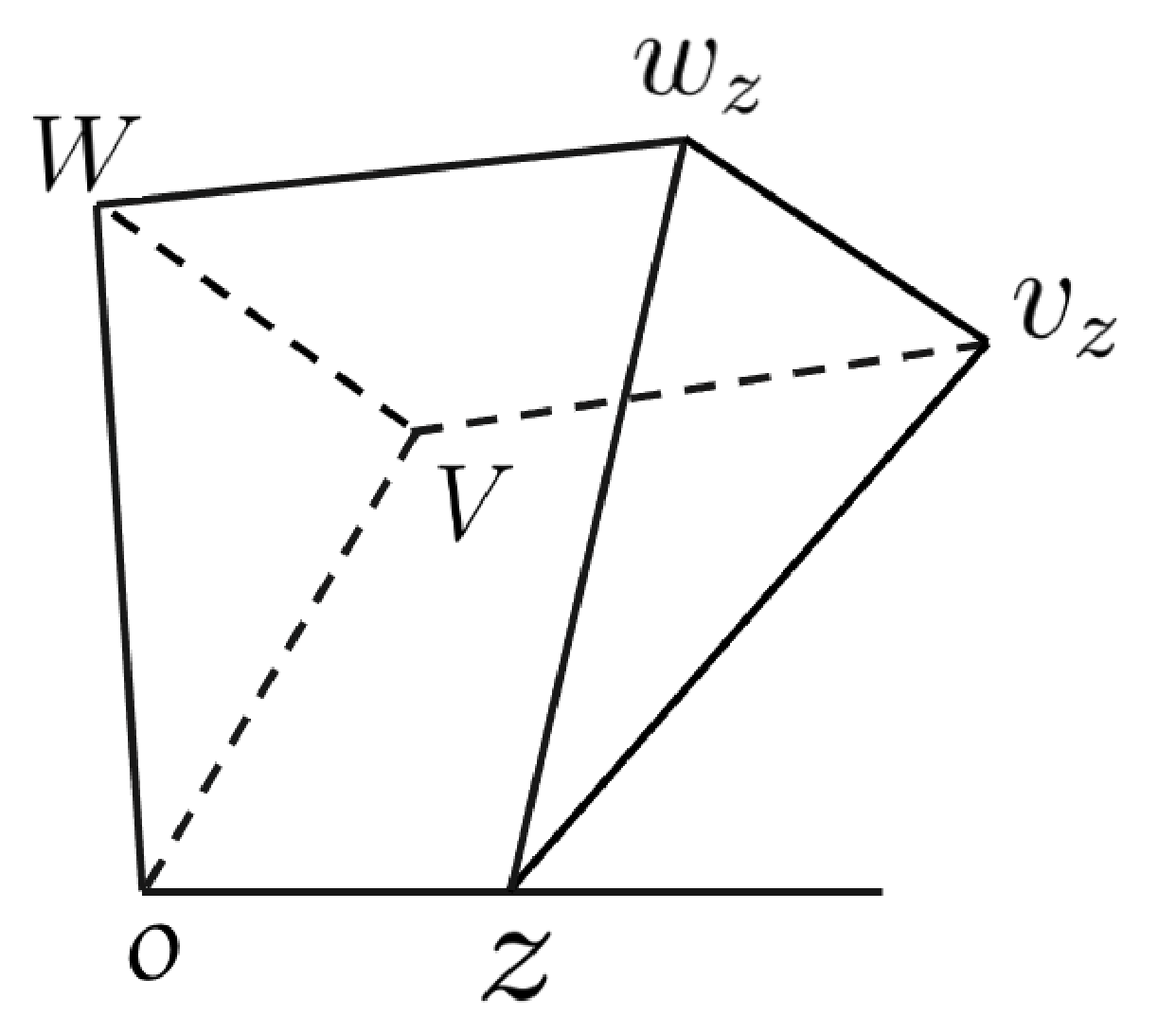}
\caption{}
\label{f:prism}
\end{figure*}
Thus as $z$ approaches~$\OO$ along the geodesic~$\exp_\OO tZ$, the
points $v_z$ and~$w_z$ converge to~$v$ and~$w$, respectively.

Since the geodesic $VW$ in~$T_\OO\XX$ does not intersect~$\IZ$, the
geodesic $v_z w_z$ does not intersect~$\OO z$ as $z \to \OO$.
Therefore the setup used to prove
Proposition~\ref{p:radial-is-isometry} (depicted in
Figure~\ref{f:radial}) is valid for any choice of lengths for $\OO z$
and~$\OO q$, even if the parameter~$s$ there grows arbitrarily large.
Thus the angle $\angle(V_z,W_z)$ and chordal distance
$\dd_\OO(v_z,w_z)$ from $v_z$ to~$w_z$ remain constant as $z$
converges to~$\OO$ along~$\exp_\OO tZ$, so $\dd_\OO(v_z,w_z) =
\dd_\OO(v,w)$.
Letting $r \to 0$ leads to the desired conclusion that $\angle(V,W) =
\angle(V_z,W_z)$.
\end{proof}

\begin{cor}\label{c:almost-iso-of-limit-log}
Let~$\MM$ be $\CAT(\kappa)$ and $\XX = \TpM$ with apex~$\OO$.  Fix
unit vectors $V,W,Z\in T_\OO\XX$ such that the geodesic in $T_\OO\XX$
from~$V$ to~$W$ does not intersect the shadow~$\IZ$ of~$Z$.  Then
$$
  \angle(V,W) = \angle(\LL_Z V,\LL_Z W)
  \quad\text{and}\quad
  \LL_Z V\,\LL_Z W = \LL_Z(VW).
$$
\end{cor}
\begin{proof}
Combine Proposition~\ref{p:almost-iso-of-singular-radial-transport}
with the Definition~\ref{d:limit-log} of limit log map~$\LL_Z$.
\end{proof}

\begin{cor}\label{c:cont-of_LL}
In the setting of Definition~\ref{d:limit-log}, the limit log map
$\LL_Z$ is continuous.
\end{cor}
\begin{proof}
It follows from Remark~\ref{r:parallel-in-shadow} when $W$ is in the
interior of~$\IZ$ and
Proposition~\ref{p:almost-iso-of-singular-radial-transport} when $W
\notin\IZ$ that
$$
  \lim_{V\to W} \LL_ZW = \LL_Z W.
$$
It remains to show that, for $W$ in the boundary of~$\IZ$,
$$
  \lim_{V \to W, V \notin\IZ} \LL_Z V = \LL_Z W.
$$
That is equivalent to
$$
  \lim_{V \to W, V \notin \IZ} \pOz V = \pOz W.
$$
The point, to this end, is that when the unit vector $V \in T_\OO\XX$
approaches $\IZ$ in the proof of
Proposition~\ref{p:convex-half-strip}, the width of the flat
half-strip spanned by $\exp_\OO(tV)$ and $\exp_Z(t\pOz V)$ decreases
to~$0$.  If $\|V\| \neq 1$ then rescale.
\end{proof}

\begin{cor}\label{c:LL-preserves-angle-to-Z}
Fix $Z \in T_\OO\XX$, where $\XX = \TpM$ and $\MM$ is~$\CAT(\kappa)$.
For all~$V \!\in\nolinebreak\! \TpM$,
$$
  \angle(V,Z) = \angle(\LL_Z V,\LL_Z Z).
$$
\end{cor}
\begin{proof}
Either $V$ lies in the shadow~$\IZ$, so both angles are~$\pi$ by
Remark~\ref{r:parallel-in-shadow}, or else $\angle(Z,V) < \pi$, in
which case use
Proposition~\ref{p:almost-iso-of-singular-radial-transport}.
\end{proof}

\begin{prop}\label{p:limit-log-contracts}
Let~$\MM$ be $\CAT(\kappa)$ and $\XX = \TpM$ with apex~$\OO$.  The
limit log map is a contraction: if $\XX = \TpM$ and $Z \in\nolinebreak
T_\OO\XX$~then
$$
  \angle(\LL_Z V, \LL_Z W) \leq \angle(V,W)
$$
for any $V,W \in \XX$.
\end{prop}
\begin{proof}
First observe that
Proposition~\ref{p:almost-iso-of-singular-radial-transport} remains
true if one of the endpoints of the geodesic~$VW$ lies in the
shadow~$\IZ$ but $VW$ is otherwise disjoint from\/~$\IZ$, by using
continuity in Corollary~\ref{c:cont-of_LL} to approach that endpoint
from $VW \setminus \IZ$.

If $VW$ never enters the shadow, then the desired result is subsumed
by Proposition~\ref{p:almost-iso-of-singular-radial-transport}.  On
the other hand, if $V$ and~$W$ are unit vectors and $VW$ enters the
shadow, then so does the shortest path~$\gamma$ joining $V$ to~$W$ in
the unit tangent sphere.
The length of~$\gamma$ is $\angle(V,W)$ by definition.  But the limit
log map takes the first and last shadow points in~$\gamma$ to the same
point, namely $-\LL_Z Z$, by Remark~\ref{r:parallel-in-shadow} (or
Proposition~\ref{p:sum=pi}, if that is preferred).  Therefore,
although the limit log map preserves the lengths of the initial and
terminal segments of~$\gamma$, which occur before entering~$\IZ$
and after its final exit, the rest of $\LL_Z(\gamma)$ is shortcut by
remaining at~$-\LL_Z Z$.
\end{proof}

Continuity allows the hypothesis of
Corollary~\ref{c:almost-iso-of-limit-log} to be weakened to allow the
geodesic from $V$ to~$W$ to meet the shadow~$\IZ$ at exactly one
point, as in Figure~\ref{f:equal-angles}.
\begin{figure*}[!ht]
\centering
\includegraphics[width=0.35\textwidth]{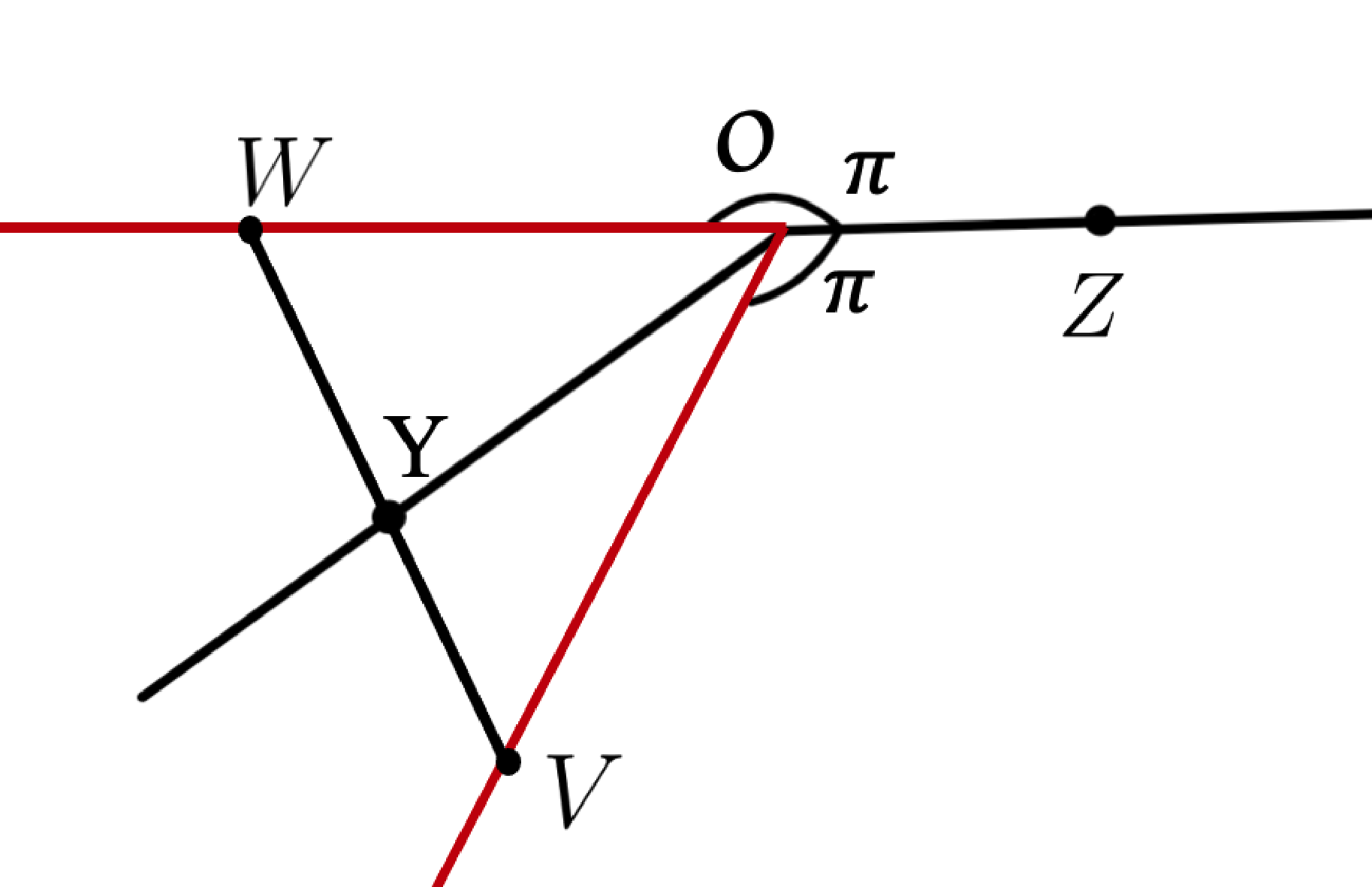}
\caption{Geodesic $VW$ meets the shadow~$\IZ$ at exactly one point.}
\label{f:equal-angles}
\end{figure*}

\begin{prop}\label{p:almost_iso-of_limit-log:v2}
Let~$\MM$ be $\CAT(\kappa)$ and $\XX = \TpM$.  For unit vectors $V,W,Z
\in T_\OO\XX$ such that every shortest path from~$V$ to~$W$
in~$T_\OO\XX$ meets $\IZ$ at exactly one point,
$$
  \dd_\OO(V,W) = \dd_Z(\LL_Z V,\LL_Z W)
  \quad\text{and}\quad
  \LL_Z V\,\LL_Z W = \LL_Z(VW),
$$
where $\dd_Z$ is the conical metric on~$\TZX$ from
Definition~\ref{d:conical-metric} applied to
Definition~\ref{d:limit-tangent-cone}.
\end{prop}
\begin{proof}
Let $Y = VW \cap \IZ$ be the single intersection point.  The result is
true if $Y = V$ or $Y = W$ by the first paragraph of the proof of
Proposition~\ref{p:limit-log-contracts}.  So break $VW$ into two
pieces: $VW = VY \cup YW$ with $Y \not\in \{V,W\}$.  Similar to the
proof of Proposition~\ref{p:almost-iso-of-singular-radial-transport},
write $z = \exp_OZ$ and $U_z = \pOz U$ for $U \in T_\OO\XX$, with
\begin{align*}
    v = \exp_\OO V,\ &\ \ \ \,v_z = \exp_zV_z,
\\  w = \exp_\OO W,  &\ \ \   w_z = \exp_zW_z.
\\  y = \exp_\OO Y,\ &\ \ \ \,y_z = \exp_zY_z.
\end{align*}

Since $Y \not\in \{V,W\}$ is the only intersection between $VW$ and
the shadow, $\angle(V,Z) <\nolinebreak \pi$ and $\angle(W,Z) < \pi$.
Applying the setup from the proof of
Proposition~\ref{p:almost-iso-of-singular-radial-transport} (see
Figure~\ref{f:prism}) and continuity of~$\LL_Z$, the parallelogram
$\OO z w_z w \subseteq \XX$ expands to a closed
half-plane~$H_W$---isometric to a half-plane in~$\RR^2$, with boundary
line spanned by~$\OO z$---that contains the ray~$\OO y$.  Similarly,
$\OO z v_z v$ expands to a closed half-plane~$H_V$.  Then $\LL_Z(UY) =
U_z Y_z$ for $U \in \{V,W\}$ by the first paragraph of the proof of
Proposition~\ref{p:limit-log-contracts} again.  The goal is to show
that $|VW| = |V_z W_z|$, or equivalenty $|vw| = |v_z w_z|$, by the
isometry $\exp_\OO: \XX \to T_\OO\XX$.

For $U \in \{V,W\}$ the half-plane $H_U$ contains a ray~$R_U$
from~$\OO$ parallel to~$YU$.  Parallel transport $\pp_{\OO \to y}$
demonstrates, by contraction in
Proposition~\ref{p:limit-log-contracts}, that $\angle(R_V,R_W) \geq
\angle(YV,YW) = \pi$, and hence $\angle(R_V,R_W) = \pi$.
For $U \in \{V,W\}$ Let
$$
  U' = R_U \cap U_z U.
$$
Then $\angle(R_V,R_W) = \pi$ means $\angle(\OO V',\OO W') = \pi$ and
$$
  |V'W'| = \norm{V'} + \norm{W'} = |VY| + |YW| = |VW|.
$$
Exponentiating at~$\OO$ yields $|v'w'| = |vw|$, so it suffices to show
$|v'w'| = |v_z w_z|$.
  
For $\ve > 0$, let $v'(\ve) = \exp_\OO\bigl(V'(\ve)\bigr)$ be a point
at distance~$\ve$ from~$v'$ along the segment~$v' v_z$ in~$\XX$.  The
ray $R_v^\ve$ from~$\OO$ through~$v'(\ve)$ has angle
$\angle(R_v^\ve,R_w) < \pi$ with the ray~$R_w = \exp_\OO(R_W)$.  By
Lemma~\ref{l:flat-triangle} the convex hull of $R_v^\ve$ and~$R_w$ is
a flat sector containing $w'$, $w_z$, $v_z$, $v'(\ve)$, and $v_z(\ve)
= \exp_z\bigl(\pp_{\OO \to z} V'(\ve)\bigr)$.  Elementary geometry in
this flat sector, using the parallelograms defined earlier, shows that
$|v'(\ve) w'| = |v_z(\ve) w_z|$ for all $\ve > 0$.  Letting $\ve \to
0$ shows that $|v' w'| = |v_z w_z|$, as desired.
\end{proof}

\begin{remark}\label{r:when-you-hit-two-shadows}
Proposition~\ref{p:almost_iso-of_limit-log:v2} is the strongest result
possible by the argument that ends the proof of
Proposition~\ref{p:limit-log-contracts}.  The proof of
Proposition~\ref{p:almost_iso-of_limit-log:v2}, if carried out to its
logical end, actually shows that the union of the half-planes $H_V$
and~$H_W$ is an isometric copy of~$\RR^2$ inside of~$\XX$.
\end{remark}

A rephrasing of Proposition~\ref{p:almost_iso-of_limit-log:v2} is the
main result of the paper.  It is one of the geometric drivers of
central limit theory for measures on
smoothly stratified metric spaces
\cite[Corollary~2.30]{tangential-collapse}.

\begin{thm}\label{t:isometry-limit-log}
Let~$\MM$ be $\CAT(\kappa)$ and $\XX = \TpM$.  If $Z \in T_\OO\XX$ and
$\KK \subseteq \TpM$ is a geodesically convex subcone containing at
most one ray in the shadow~$\IZ$, then the restriction $\LL_Z|_\KK:
\KK \to \LL_Z(\KK)$ of the limit log map along~$Z$ is an isometry
onto~its~\mbox{image}.
\end{thm}
\begin{proof}
This is a direct consequence of
Proposition~\ref{p:almost_iso-of_limit-log:v2}.
\end{proof}

\begin{remark}\label{r:isometry-limit-log}
Theorem~\ref{t:isometry-limit-log} summarizes the limit log
map~$\LL_Z$ the following way: $\LL_Z$ collapses the shadow $\IZ$ to a
single ray (Remark~\ref{r:parallel-in-shadow}) while preserving the
rest of~$\XX = \TpM$ isometrically.  Of course, any part of any
geodesic that passes through the shadow collapses to a segement along
the ray that is the collapsed image under~$\LL_Z$ of the shadow, but
all geodesics otherwise maintain their integrity.
\end{remark}

A~simple consequence of Theorem~\ref{t:isometry-limit-log} is
generally useful and arises while manufacturing Gaussian-distributed
vectors on the tangent cone of a smoothly stratified $\CAT(\kappa)$
space~$\MM$ \cite[Section~6.1]{escape-vectors}.

\begin{cor}\label{c:proper}
Let~$\MM$ be $\CAT(\kappa)$ and $\XX = \TpM$.  If $Z \in T_\OO\XX$
then the limit log along~$Z$ is a proper mapping.
\end{cor}
\begin{proof}
Limit log preserves distance from the apex by
Theorem~\ref{t:isometry-limit-log} (see
Remark~\ref{r:isometry-limit-log}) and $\MM$ is locally compact.
\end{proof}

The final less elementary consequence of
Theorem~\ref{t:isometry-limit-log} is key to preservation of
fluctuating cones under limit log maps in subsequent work
\cite[Corollary~2.27]{tangential-collapse}.
It requires a simple definition.

\begin{defn}[Hull]\label{d:hull}
Given a subset $\cS \subseteq \XX$ of a $\CAT(0)$ conical space~$\XX$,
the \emph{hull} of~$\cS$ is the smallest geodesically convex cone
$\hull\cS \subseteq \XX$ containing~$\cS$.
\end{defn}

\begin{cor}\label{c:hull-preserved-by-limit-log}
Let~$\MM$ be $\CAT(\kappa)$ and $\XX = \TpM$.  For any $Z \in
T_\OO\XX$, taking limit log along~$Z$ subcommutes with taking convex
cones: for any subset $\cS \subseteq \XX$,
$$
  \LL_Z(\hull\cS) \subseteq \hull\LL_Z(\cS).
$$
\end{cor}
\begin{proof}
If $\hull(\LL_Z\cS)$ contains no positive-length vector
in~$\LL_Z\bigl(\IZ\bigr)$, then $\hull\cS$ itself contains no
positive-length vector in the shadow~$\IZ$, because the lift of any
shortest path not meeting~$\IZ$ is a shortest path.  Indeed, by
Theorem~\ref{t:isometry-limit-log} the preimage under~$\LL_Z$ of any
shortest path not meeting~$\IZ$ is a candidate geodesic between the
preimage endpoints whose length equals the distance between the
preimage endpoints, because limit log is a contraction by
Proposition~\ref{p:limit-log-contracts}.

This reduces the question to the case where $\hull(\LL_Z\cS)$ contains
a positive-length vector in~$\LL_Z\bigl(\IZ\bigr)$.  Let $\cS'$ be the
union of all shortest paths in~$\XX$ between pairs of points of~$\cS$.
Iterating if necessary, it suffices to prove $\LL_Z(\cS') \subseteq
\hull\LL_Z(\cS)$.  Suppose that $pq$ is a shortest path in~$\XX$
between $p \in \cS$ and~$q \in \cS$.  Then $pq$ breaks into a union of
closed shortest paths each either contained in~$\IZ$ or having at most
one endpoint in~$\IZ$.  By Theorem~\ref{t:isometry-limit-log},
applying~$\LL_Z$ to each of these geodesic segments yields either a
segment in~$\LL_Z\bigl(\IZ\bigr)$, which is contained
in~$\hull\LL_Z(\cS)$ by hypothesis, or a shortest path from a point
of~$\LL_Z(\cS)$ to~$\LL_Z\bigl(\IZ\bigr)$, which is also contained
in~$\hull\LL_Z(\cS)$.
%
%
%
\end{proof}




\end{document}